\documentclass[11pt]{article}
\usepackage[english]{babel}
\usepackage{times}
\usepackage{amsmath,amstext,amssymb,amsthm, amsfonts}
\usepackage[dvips]{graphics,color}
\newcommand{\bi}{\begin{itemize}}  %begin itemize
\newcommand{\ei}{\end{itemize}}     %end itemize
\newcommand{\bc}{\begin{center}}  %begin center
\newcommand{\ec}{\end{center}}     %end center

\newcommand{\ls}[1]
   {\dimen0=\fontdimen6\the\font \lineskip=#1\dimen0
   \advance\lineskip.5\fontdimen5\the\font \advance\lineskip-\dimen0
   \lineskiplimit=.9\lineskip \baselineskip=\lineskip
   \advance\baselineskip\dimen0 \normallineskip\lineskip
   \normallineskiplimit\lineskiplimit \normalbaselineskip\baselineskip
   \ignorespaces }

\numberwithin{equation}{section}

\newtheorem{lemma}{Lemma}[section]
\newtheorem{theorem}[lemma]{Theorem}
\newtheorem{corollary}[lemma]{Corollary}
\newtheorem{definition}[lemma]{Definition}

\newtheorem{example}[lemma]{Example}

\newtheorem{remark}[lemma]{Remark}

\def\X{\mathbb{X}}
\def\U{\mathbb{U}}
\def\K{\mathbb{K}}
\def\Y{\mathbb{Y}}
\def\S{\mathbb{S}}
\def\R{\mathbb{R}}

\def\F{\mathbb{F}}

\title{Berge's Theorem for Noncompact Image Sets}

\begin{document}
\date{}
\maketitle

\begin{center}
  Eugene~A.~Feinberg \footnote{Department of Applied Mathematics and
Statistics,
 State University of New York at Stony Brook,
Stony Brook, NY 11794-3600, USA, eugene.feinberg@sunysb.edu}, \
Pavlo~O.~Kasyano${\rm v}^2$, %\footnote{Institute for Applied System Analysis,
%National Technical University of Ukraine ``Kyiv Polytechnic
%Institute'', Peremogy ave., 37, build, 35, 03056, Kyiv, Ukraine,\
%kasyanov@i.ua.},
 and Nina~V.~Zadoianchuk\footnote{Institute for
Applied System Analysis, National Technical University of Ukraine
``Kyiv Polytechnic Institute'', Peremogy ave., 37, build, 35,
03056, Kyiv, Ukraine, kasyanov@i.ua, ninellll@i.ua.
} \\

\smallskip
%July 3, 2007
\end{center}

\begin{abstract}
For an upper semi-continuous set-valued mapping from one
topological space to another and for a lower semi-continuous
function defined on the product of these spaces, Berge's theorem
states lower semi-continuity of the minimum of this function taken
over the image sets.  It assumes that the image sets are compact.
For Hausdorff topological spaces, this paper extends Berge's
theorem to set-valued mappings with possible noncompact image sets
and studies relevant properties of minima.
%
%Berge's Theorem states lower semi-continuity of minima function
%
%Let $\X$ and $\Y$ be Hausdorff topological spaces, $\Phi$ be a
%multifunction from $\X$ to $\Y$, and $u(\cdot,\cdot)$ be a
%$\overline{\R}$-valued function on graph $\Phi$. We provide
%$\K$-inf-compactness property for $u(\cdot,\cdot)$ on graph $\Phi$,
%that for continuity properties of the function
%$v(x):=\inf\limits_{y\in {\Phi}(x)}u(x,y)$, $x\in\X$, when this
%function may be unbounded and minimizers $\Phi(x)$ may be
%noncompact. This conditions assumptions.
\end{abstract}

\section{Introduction} Let $\X$ and $\Y$ be
Hausdorff topological spaces, $u:\X\times \Y\to
\overline{\mathbb{R}}=\mathbb{R}\cup\{\pm\infty\}$ and
$\Phi:\X\to2^{\Y}\setminus \{\emptyset\}$. For $x\in\X$ define
\begin{equation}\label{eq:lm10}
v(x):=\inf\limits_{y\in {\Phi}(x)}u(x,y).
\end{equation}
For Hausdorff topological spaces, the well-known Berge's Theorem
(cf. Berge \cite[Theorem~2, p.~116]{Ber}) has the following formulation.\\
\vskip -0.5 cm {\bf Berge's Theorem {\rm (Hu and
Papageorgiou~\cite[Proposition~3.3, p.~83]{Hu}}).} \textit{If
$u:\X\times\Y\to\overline{\mathbb{R}}$ is a lower semi-continuous
function and $\Phi:\X\to 2^{\Y}\setminus\{\emptyset\}$ is a
compact-valued upper semi-continuous set-valued mapping, then the
function $v:\X\to\overline{\mathbb{R}}$ is lower semi-continuous.}

%\vskip 0.2 cm

Luque-V\'asquez and Hern\'andez-Lerma \cite{LVHL} provide an
example of a continuous $\Phi$ with possible noncompact sets
$\Phi(x)$ and of a lower semi-continuous function $u(x,y)$ being
inf-compact in $y$, when $v(x)$ is not lower semi-continuous. In
this paper, we extend Berge's theorem to possibly noncompact sets
$\Phi(x),$ $x\in\X.$ Let ${\rm Gr}_Z(\Phi)=\{(x,y)\in
Z\times\Y\,:\, y\in\Phi(x)\}$, where $Z\subseteq\X.$ For a
topological space $\U$, we denote by $\K(\U)$ \emph{ the family of
all nonempty compact subsets
of     \ $\U.$} %The following definition plays an important role in this
%paper.

\begin{definition}
A function $u:\X\times \Y\to \overline{\mathbb{R}}$ is called
$\K$-inf-compact on ${\rm Gr}_{\X}(\Phi)$, %, if for any nonempty compact subset
%$K\subseteq\X$,
if for every $K\in \K(\X)$ this function is inf-compact on  ${\rm
Gr}_K(\Phi)$.
\end{definition}

The following theorem is the main result of this paper.

\begin{theorem}\label{MT1}
If the function $u:\X\times \Y\to \overline{\mathbb{R}}$ is
$\K$-inf-compact on ${\rm Gr}_{\X}(\Phi)$, then the function
$v:\X\to\overline{\mathbb{R}}$ is lower semi-continuous.
\end{theorem}

\section{Properties of $\K$-inf-compact Functions and Proof of Theorem~\ref{MT1}}
 For an $\overline{\mathbb{R}}$-valued
function $f$, defined on a nonempty subset $U$ of a topological
space $\mathbb{U},$ consider the level sets
\begin{equation}\label{def-D}
\mathcal{D}_f(\lambda;U)=\{y\in U \, : \,  f(y)\le
\lambda\},\qquad \lambda\in\R.\end{equation}  We recall that a
function $f$ is \textit{lower semi-continuous on $U$} if all the
level sets $\mathcal{D}_f(\lambda;U)$
 are closed, and a function $f$ is
\textit{inf-compact  on $U$} if all these sets are
compact.%; Feinberg and Lewis \cite{FL}. %The level sets $\mathcal{D}_f(\lambda;U)$ satisfy the
%following properties that are used in this paper:
%
%(a) if $\lambda_1>\lambda$ then $\mathcal{D}_f(\lambda;U)\subseteq
%\mathcal{D}_f(\lambda_1;U);$
%
%(b) if $g,f$ are functions on $U$ satisfying $g(y)\ge f(y)$ for
%all $y\in U$ then $\mathcal{D}_g(\lambda;U)\subseteq
%\mathcal{D}_f(\lambda;U).$
%
%For any subset $K$ of $\X$ we consider
%\[
%{\rm Gr}_K(F):=\{(x,y)\,:\, x\in K,\,y\in F(x) \}.
%\]
%When $K=\X$ define the graph of $F$ by ${\rm Gr}_{\X} (F):={\rm
%Gr}_\X(F)$.
%
%\begin{remark} We note that $\K$-inf-compactness of
%$u(\cdot,\cdot)$ on ${\rm Gr}_{\X}(\Phi)$ means that for any $K\in
%\K(\X)$ the function $u(\cdot,\cdot)$ is inf-compact on ${\rm
%Gr}_K(\Phi)$.
%\end{remark}
%
%Now we provide some sufficient conditions for $\K$-inf-compactness
%of $u(\cdot,\cdot)$ on ${\rm Gr}_{\X}(\Phi)$, including Berge's
%assumptions.

For an upper semi-continuous set-valued mapping ${\Phi}:\X\to
K(\Y)$, the set ${\rm Gr}_\X(\Phi)$ is closed; Berge~\cite[Theorem
6, p. 112]{Ber}. Therefore, for such $\Phi$, if a function
$u(\cdot,\cdot)$ is lower semi-continuous on $\X\times\Y$, then it
is lower semi-continuous on ${\rm Gr}_\X(\Phi)$. Thus,
Lemma~\ref{lm0}(i) implies that Theorem~\ref{MT1} is a natural
generalization of Berge's Theorem; see also Remark~\ref{r.2.4} and
Lemma~\ref{lm4}.
%
%If ${\Phi}:\X\to K(\Y)$ is upper semi-continuous, and $u$ is lower
%semi-continuous on $\X\times\Y$,
%
%
% and on ${\rm
%Gr}_\X(\Phi)$ are equivalent.  Indeed, for an upper
%semi-continuous set-valued mapping $\Phi$, the set ${\rm
%Gr}_\X(\Phi)$ is closed Berge~\cite[Theorem 6, p. 112]{Ber}. Thus,
%if $u$ is lower semi-continuous on $\X\times\Y$, it is lower
%semi-continuous on ${\rm Gr}_\X(\Phi)$. Also, a lower
%semi-continuous on ${\rm Gr}_\X(\Phi)$ function $u$ can be
%continued to $\X\times\Y$ by setting $u(x,y)=+\infty$, when
%$(x,y)\in \X\times\Y\setminus {\rm Gr}_\X(\Phi).$  The new
%function is lower semi-continuous on $\X\times\Y$, and the
%function $v$ remains unchanged. Lemma~\ref{lm0}(i) implies that
%Theorem~\ref{MT1} generalizes Berge's Theorem.

\begin{lemma}\label{lm0}
The following statements hold:

(i) if $u:\X\times\Y\to \overline{\mathbb{R}}$ is lower
semi-continuous on ${\rm Gr}_\X ({{\Phi}})$ and ${\Phi}:\X\to
K(\Y)$ is upper semi-continuous, then the function
$u(\cdot,\cdot)$ is $\K$-inf-compact on ${\rm Gr}_\X ({{\Phi}})$;

(ii) if $u:\X\times\Y\to \overline{\mathbb{R}}$ is inf-compact on
${\rm Gr}_\X(\Phi)$, then the function $u(\cdot,\cdot)$ is
$\K$-inf-compact on ${\rm Gr}_\X(\Phi)$.
\end{lemma}
\begin{proof} (i) Let $u:\X\times\Y\to \overline{\mathbb{R}}$ be lower
semi-continuous on ${\rm Gr}_{\X} ({{\Phi}})$ and ${\Phi}:\X\to
K(\Y)$ be upper semi-continuous. For $K\in\K(\X)$ and $\lambda\in
\R$, the level set $\mathcal{D}_{u(\cdot,\cdot)}(\lambda;{\rm
Gr}_K(\Phi))$ is compact. Indeed, since $K$ is a compact set and
the set-valued mapping ${\Phi}:\X\to K(\Y)$ is upper
semi-continuous, then Berge \cite[Theorem~3 on p. 110]{Ber}
implies that the image ${\Phi}(K)$ is also compact. Thus, $K\times
\Phi(K)$ is a  compact subset of $\X\times \Y$. In virtue of lower
semi-continuity of $u(\cdot,\cdot)$ on ${\rm Gr}_{\X}(\Phi),$ the
level set $\mathcal{D}_{u(\cdot,\cdot)}(\lambda;{\rm
Gr}_{\X}(\Phi))$ is closed. Therefore the level set
$\mathcal{D}_{u(\cdot,\cdot)}(\lambda;{\rm
Gr}_K(\Phi))=(K\times\Phi(K))\cap
\mathcal{D}_{u(\cdot,\cdot)}(\lambda;{\rm Gr}_{\X}(\Phi))$ is
compact.

(ii) Let $u(\cdot,\cdot)$ be an inf-compact function on ${\rm
Gr}_{\X}(\Phi)$, $K\in\K(\X)$, and $\lambda\in \R$. Since
$K\times\Y$ is a closed subset of $\X\times\Y$ and the set
$\mathcal{D}_{u(\cdot,\cdot)}(\lambda;{\rm Gr}_{\X}(\Phi))$ is
compact, the level set $\mathcal{D}_{u(\cdot,\cdot)}(\lambda;{\rm
Gr}_K(\Phi))=(K\times\Y)\cap
\mathcal{D}_{u(\cdot,\cdot)}(\lambda;{\rm Gr}_{\X}(\Phi))$ is
compact.
\end{proof}

\begin{lemma}\label{lm00101}
If $u(\cdot,\cdot)$ is $\K$-inf-compact function on ${\rm
Gr}_{\X}(\Phi)$, then for every $x\in\X$ the function $u(x,\cdot)$
is inf-compact on ${\Phi}(x)$.
\end{lemma}
\begin{proof}
For an arbitrary $\lambda\in\R$ and an arbitrary fixed $x\in \X$,
consider the set
\[
\mathcal{D}_{u(x,\cdot)}(\lambda;\Phi(x))=\{y\in {\Phi}(x)\,:\,
u(x,y)\le\lambda\}= \mathcal{D}_{u(\cdot,\cdot)}(\lambda;{\rm
Gr}_{\{x\}}(\Phi)).
\]
$\K$-inf-compactness of $u(\cdot,\cdot)$ on ${\rm Gr}_{\X}(\Phi)$
implies, that this set is compact.
\end{proof}

\begin{lemma}\label{lsc}
A $\K$-inf-compact function $u(\cdot,\cdot)$ on ${\rm
Gr}_{\X}(\Phi)$ is lower semi-continuous on ${\rm Gr}_{\X}(\Phi)$.
\end{lemma}
\begin{proof}
Let $\lambda\in\R.$ We need to show that the level set
$\mathcal{D}_{u(\cdot,\cdot)}(\lambda;{\rm Gr}_{\X}(\Phi))$ is
closed. If this is not true, according to Hu and Papageorgiou
\cite[Proposition~A.1.24(b), p.~893]{Hu}, there exists a net
$(x_\alpha,y_\alpha)\to (x,y)$ in $\X\times\Y$ with
$(x_\alpha,y_\alpha)\in \mathcal{D}_{u(\cdot,\cdot)}(\lambda;{\rm
Gr}_{\X}(\Phi))$ for any $\alpha$, such that $(x,y)\notin
\mathcal{D}_{u(\cdot,\cdot)}(\lambda;{\rm Gr}_{\X}(\Phi))$. On the
other hand, the set $K=(\cup_\alpha\{x_\alpha\})\cup\{x\}$ is
compact. Thus, by $\K$-inf-compactness of $u(\cdot,\cdot)$ on
${\rm Gr}_{\X}(\Phi)$, the level set
$\mathcal{D}_{u(\cdot,\cdot)}(\lambda;{\rm Gr}_K(\Phi))$ is
compact too. As $\{(x_\alpha,y_\alpha)\}_\alpha\subseteq
\mathcal{D}_{u(\cdot,\cdot)}(\lambda;{\rm Gr}_K(\Phi))$, then
$(x,y)\in \mathcal{D}_{u(\cdot,\cdot)}(\lambda;{\rm Gr}_K(\Phi))$.
This is a contradiction. Therefore, $u(\cdot,\cdot)$ is lower
semi-continuous on ${\rm Gr}_{\X}(\Phi)$.
\end{proof}

\begin{remark}\label{r.2.4}
Lower semi-continuity of $u(\cdot,\cdot)$ on ${\rm
Gr}_{\X}(\Phi)$, inf-compactness of the function $u(x,\cdot)$ on
${\Phi}(x)$ for every $x\in\X$, and continuity of $\Phi:\X\to
2^{\Y}\setminus\{\emptyset\}$ do not imply lower semi-continuity
of $v(\cdot)$; Luque-Vasques and Hern\'{a}ndez-Lerma \cite{LVHL}.
$\K$-inf-compactness of $u(\cdot,\cdot)$ on ${\rm Gr}_{\X}(\Phi)$
is an assumption that is close to the mentioned conditions, and it
guaranties the lower semi-continuity of $v(\cdot)$. As follows
from Feinberg and Lewis \cite[Proposition~3.1]{FL},  where Polish
$\X$ and $\Y$ are considered, if $u(\cdot,\cdot)$ is inf-compact
on ${\rm Gr}_\X(\Phi)$, then $v(\cdot)$ is inf-compact.
\end{remark}

The following lemma provides a useful criterium for
$\K$-inf-compactness of $u(\cdot,\cdot)$ on ${\rm Gr}_{\X}(\Phi)$,
when the spaces $\X$ and $\Y$ are metrizable. In this form the
$\K$-inf-compactness assumption is introduced in Feinberg,
Kasyanov and Zadoianchuk \cite{arx} as Assumption ${\rm\bf
(W^*)}$(ii).

\begin{lemma}\label{lm4}
Let $\X$ and $\Y$ be metrizable spaces. Then $u(\cdot,\cdot)$ is
$\K$-inf-compact on ${\rm Gr}_{\X}(\Phi)$ if and only if the
following two conditions hold:

(i) $u(\cdot,\cdot)$ is lower semi-continuous on ${\rm
Gr}_{\X}(\Phi)$;

(ii) if a sequence $\{x_n \}_{n=1,2,\ldots}$ with values in $\X$
converges and its limit $x$ belongs to $\X$ then any sequence
$\{y_n \}_{n=1,2,\ldots}$ with $y_n\in \Phi(x_n)$, $n=1,2,\ldots,$
satisfying the condition that the sequence $\{u(x_n,y_n)
\}_{n=1,2,\ldots}$ is bounded above, has a limit point $y\in
\Phi(x).$
\end{lemma}
\begin{proof}
Let $u(\cdot,\cdot)$ be $\K$-inf-compact on ${\rm Gr}_{\X}(\Phi)$.
Then, by Lemma~\ref{lsc}, $u(\cdot,\cdot)$ is lower
semi-continuous on ${\rm Gr}_{\X}(\Phi)$. Thus (i) holds. Consider
a convergent sequence $\{x_n \}_{n=1,2,\ldots}$ with values in
$\X$, such that its limit $x$ belongs to $\X$. Moreover, let a
sequence $\{y_n \}_{n=1,2,\ldots}$ with $y_n\in \Phi(x_n)$,
$n=1,2,\ldots,$ satisfy the condition that the sequence
$\{u(x_n,y_n) \}_{n=1,2,\ldots}$ is bounded above by some
$\lambda\in\R$. Then the set $K=(\cup_{n\ge 1}\{x_n\})\cup\{x\}$
is compact. Since $u(\cdot,\cdot)$ is inf-compact on ${\rm
Gr}_{K}(\Phi)$, then the sequence $\{(x_n,y_n)\}_{n=1,2,\ldots}$
belongs to the compact set
$\mathcal{D}_{u(\cdot,\cdot)}(\lambda;{\rm Gr}_K(\Phi))$.
Therefore this sequence has a limit point $y\in \Phi(x).$ Thus
(ii) holds.

Let (i) and (ii) hold. Fix $K\in\K(\X)$ and $\lambda\in \R$. The
level set $\mathcal{D}_{u(\cdot,\cdot)}(\lambda;{\rm Gr}_K(\Phi))$
is compact. Indeed let $\{(x_n,y_n)\}_{n\ge 1}\subseteq
\mathcal{D}_{u(\cdot,\cdot)}(\lambda;{\rm Gr}_K(\Phi))$. Since $K$
is a compact set, the sequence $\{x_n\}_{n=1,2,\ldots}$ has a
subsequence $\{x_{n_k}\}_{k=1,2,\ldots}$ that converges to its
limit point $x\in \X$. By condition (ii), $\{y_{n_k}
\}_{k=1,2,\ldots}$ has a limit point $y\in \Phi(x)$, that is
$(x,y)\in{\rm Gr}_{\X}(\Phi)$ is a limit point for the sequence
$\{(x_n,y_n)\}_{n\ge 1}\subseteq
\mathcal{D}_{u(\cdot,\cdot)}(\lambda;{\rm Gr}_K(\Phi))$. Lower
semi-continuity of $u(\cdot,\cdot)$ on ${\rm Gr}_{\X}(\Phi)$
implies the inequality $u(x,y)\le \lambda$. Thus $(x,y)\in
\mathcal{D}_{u(\cdot,\cdot)}(\lambda;{\rm Gr}_K(\Phi))$, and the
level set $\mathcal{D}_{u(\cdot,\cdot)}(\lambda;{\rm Gr}_K(\Phi))$
is compact.
\end{proof}

\begin{proof}[Proof of Theorem~\ref{MT1}]
Let $v(\cdot)$ be not lower semi-continuous. Then the set
$\mathcal{D}_{v(\cdot)}(\lambda;\X)=\{x\in\X\,:\, v(x)\le
\lambda\}$ is not closed for some $\lambda\in\R$. According to Hu
and Papageorgiou \cite[Proposition~A.1.24(b), p.~893]{Hu}, there
exists a net $x_\alpha\to x$ in $\X$ with $x_\alpha\in
\mathcal{D}_{v(\cdot)}(\lambda;\X)$ for any $\alpha$, such that
$x\notin \mathcal{D}_{v(\cdot)}(\lambda;\X)$. On the other hand,
for any $\alpha$ there exists $y_\alpha\in \Phi(x_\alpha)$ such
that $v(x_\alpha)=u(x_\alpha,y_\alpha)$. Consider a compact set
$K=(\cup_\alpha\{x_\alpha\})\cup\{x\}$. In virtue of
$\K$-inf-compactness of $u(\cdot,\cdot)$ on ${\rm Gr}_{\X}(\Phi)$,
the level set $\mathcal{D}_{u(\cdot,\cdot)}(\lambda;{\rm
Gr}_K(\Phi))$ is compact. Moreover,
$\{(x_\alpha,y_\alpha)\}\subseteq
\mathcal{D}_{u(\cdot,\cdot)}(\lambda;{\rm Gr}_K(\Phi))$ for any
$\alpha$. Thus, $(x,y)\in\mathcal{D}_{u(\cdot,\cdot)}(\lambda;{\rm
Gr}_K(\Phi))$ for some $y\in \Phi(x)$ and, therefore, $x\in
\mathcal{D}_{v(\cdot)}(\lambda;\X)$. This is a contradiction.
\end{proof}

\section{Additional Properties of Minima}

Throughout this section $L(\mathbb{X})$ denotes the class of all
lower semi-continuous functions
$\varphi:\mathbb{X}\to\overline{\mathbb{R}}$ with ${\rm
dom\,}\varphi:=\{x\in\X\, : \, \varphi(x)\ne\infty\}\ne
\emptyset$. For a topological space $\U$, let ${\mathcal B}(\U)$
be a Borel $\sigma$-field on $\U$, that is, the $\sigma$-field
generated by all open sets of the space $\U$. For a set $E\subset
\U$, we denote by ${\mathcal B}(E)$ the $\sigma$-field whose
elements are intersections of $E$ with elements of ${\mathcal
B}(\U)$.  Observe that $E$ is a topological space with induced
topology from $\U$, and ${\mathcal B}(E)$ is its Borel
$\sigma$-field.

\begin{theorem}\label{MT1a}
If a function $u(\cdot,\cdot)$ is $\K$-inf-compact on ${\rm
Gr}_{\X}(\Phi)$, then the infimum in (\ref{eq:lm10}) can be
replaced with the minimum, and the nonempty sets ${\Phi}^*(x)$,
$x\in\X$, defined as
\begin{equation}\label{eq:lm000}
{\Phi}^*(x)=\left\{a\in {\Phi}(x):\,v(x)=u(x,a)\right\},
\end{equation}
satisfy the following properties:

(a) the graph ${\rm Gr}_{\X}({\Phi}^*)=\{(x,y):\, x\in\X, y\in
{\Phi}^*(x)\}$ is a Borel subset of \  $\X\times \Y$;

 (b) if $v(x)=+\infty$,
then ${\Phi}^*(x)={\Phi}(x)$, and, if $v(x)<+\infty$, then
${\Phi}^*(x)$ is compact.
\end{theorem}
\begin{proof}
$\K$-inf-compactness of $u(\cdot,\cdot)$ on ${\rm Gr}_{\X}(\Phi)$
implies that infimum in (\ref{eq:lm10}) can be replaced with the
minimum. This follows from Lemmas~\ref{lm00101}, \ref{lsc}, and
the classical extreme value theorem.

Consider the nonempty sets ${\Phi}^*(x)$, $x\in\X$, defined in
(\ref{eq:lm000}). The graph ${\rm Gr}_{\X}({\Phi}^*)$ is a Borel
subset of $\X\times \Y$, because ${\rm
Gr}_{\X}({\Phi}^*)=\{(x,y)\,: v(x)=u(x,y) \}$, and the functions
$u(\cdot,\cdot)$ and $v(\cdot)$ are lower semi-continuous on ${\rm
Gr}_{\X}({\Phi})$ and $\X$ respectively (Theorem~\ref{MT1}), and
therefore they are Borel.

 We remark that, if $v(x)=+\infty$, then
${\Phi}^*(x)={\Phi}(x)$. If $v(x)<+\infty$, then
Lemma~\ref{lm00101} implies that the set ${\Phi}^*(x)$ is compact.
Indeed, fix any $x\in\X_v:=\{x\in\X\,:\, v(x)<+\infty\}$ and set
$\lambda=v(x)$. Then the set $ {\Phi}^*(x)=\{y\in {\Phi}(x)\,: \,
u(x,y)\le\lambda \}=\mathcal{D}_{u(x,\cdot)}(\lambda)$ is compact,
because $u(x,\cdot)$ is inf-compact on ${\Phi}(x)$.
\end{proof}

\begin{corollary}\label{cor1} (cf.  Feinberg and Lewis
\cite[Proposition~3.1]{FL})
If a function $u:\X\times\Y\to\overline{\mathbb{R}}$ is
inf-compact on ${\rm Gr}_{\X}(\Phi)$, then the function
$v:\X\to\overline{\R}$ is inf-compact and the conclusions of
Theorem \ref{MT1a} hold.
\end{corollary}
\begin{proof}
The conclusions of Theorem \ref{MT1a} directly follow from
Lemma~\ref{lm0}(ii) and Theorems~\ref{MT1}, \ref{MT1a}. The
function $v(\cdot)$ is inf-compact, since any level set
$\mathcal{D}_{v(\cdot)}(\lambda;\X)$ is compact as the projection
of the compact set $\mathcal{D}_{u(\cdot,\cdot)}(\lambda;{\rm
Gr}_\X(\Phi))$ on $\X$, $\lambda\in\R$.
\end{proof}

Let $\mathbb{F}=\{\phi:\X\to \Y \, : \,   \phi \mbox{ is Borel
and\ } \phi(x)\in {\Phi}(x)\ {\rm for\ all\ } x\in \X\}.$ A
 mapping $\phi\in \mathbb{F}$, is called a selector (or
a measurable selector).

\begin{theorem}\label{MT2}
Let $\X$ and $\Y$ be \textit{Borel subsets} of Polish (complete
separable metric) spaces and $u(\cdot,\cdot)$ be $\K$-inf-compact
on ${\rm Gr}_{\X}(\Phi)$. Then there exists a selector $f\in
\mathbb{F}$ such that
\begin{equation}\label{eq:lm00}
v(x)=u(x,f(x)),\qquad x\in \X.
\end{equation}
\end{theorem}
\begin{proof}
Let us prove the existence of $f\in\mathbb{F}$ satisfying
(\ref{eq:lm00}). Since the function $v(\cdot)$ is lower
semi-continuous (Theorem~\ref{MT1}), it is Borel and the sets
$\X_\infty:=\{x\in\X\,:\, v(x)=+\infty\}$ and $\X_v=\X\setminus
\X_\infty$ are Borel. Therefore, the ${\rm Gr}_{\X_v}(\Phi^*)$ is
the Borel subset of ${\rm Gr}_{\X}({\Phi}^*)\setminus
(\X_\infty\times\Y)$. Since the nonempty sets ${\Phi}^*(x)$ are
compact for all $x\in\X_v$, the Arsenin-Kunugui Theorem (cf.
Kechris~\cite[p. 297]{Kec}) implies the existence of a Borel
selector $f_1:\, \X_v\to \Y$ such that $f_1(x)\in {\Phi}^*(x)$ for
all $x\in\X_v.$ Consider any Borel mapping $f_2$ from $\X$ to $\Y$
satisfying $f_2(x)\in {\Phi}(x)$ for all $x\in\X$ and set
\[
f(x)=\begin{cases} f_1(x), &{\rm if\ } x\in\X_v,\\
f_2(x), &{\rm if\ } x\in\X_\infty.
\end{cases}
\]
Then $f\in\F$ and $f(x)\in {\Phi}^*(x)$ for all $x\in\X.$
\end{proof}

\section{Continuity of Minima}
For a set $U$, denote by $\S(U)$  \textit{the family of all
nonempty subsets} of $U.$ A set-valued mapping ${F}:\X \to \S(\Y)$
is \textit{upper semi-continuous} at $x\in\X$ if, for any
neighborhood $\mathcal{G}$ of the set $F(x)$, there is a
neighborhood of $x$, say $\mathcal{O}(x)$, such that
$F(y)\subseteq \mathcal{G}$ for all $y\in \mathcal{O}(x)$; a
set-valued mapping ${F}:\X \to \S(\Y)$ is \textit{lower
semi-continuous} at $x\in\X$ if, for any neighborhood
$\mathcal{G}$ of the set $F(x)$, there is a neighborhood of $x$,
say $\mathcal{O}(x)$, such that if $y\in \mathcal{O}(x)$, then
$F(y)\cap \mathcal{G}\ne\emptyset$; see e.g., Berge
\cite[p.~109]{Ber} or Zgurovsky et al. \cite[Chapter~1,
p.~7]{ZMK1}. A set-valued mapping is called \textit{upper (lower)
semi-continuous}, if it is upper (lower) semi-continuous at all
$x\in\X$.

Throughout this section we assume that $u(\cdot, \cdot)$ is a real
function, that is $u:\X\times\Y\to \R$.

\begin{theorem}\label{MT3}
If $u(\cdot,\cdot)$ is a $\K$-inf-compact, %on ${\rm Gr}_{\X}(\Phi)$,
%$u(\cdot,\cdot)$ is
continuous function on ${\rm Gr}_{\X}(\Phi)$ and $\Phi:\X\to
\S(\Y)$ is lower semi-continuous, then the function $v(\cdot)$,
defined in (\ref{eq:lm10}), is continuous on $\X$ and the solution
multifunction $\Phi^*:\X\to \K(\Y)$ has a closed graph. If,
moreover, $\Phi$ is upper semi-continuous, then $\Phi^*$ is upper
semi-continuous.
\end{theorem}
\begin{proof}
By Theorem~\ref{MT1}, the function $v(\cdot)$, defined in
(\ref{eq:lm10}), belongs to $L(\X)$. Moreover, $v(x)<+\infty$ and,
 according to Theorem~\ref{MT1a},
$\Phi^*(x)\in \K(\Y)$ for all $x\in \X.$ %${\Phi}^*(x)$ is compact
%for any $x\in \X$.
 Lower semi-continuity of $\Phi:\X\to \S(\Y)$,
upper semi-continuity of $u(\cdot,\cdot)$ on ${\rm
Gr}_{\X}(\Phi)$, and Hu and Papageorgiou \cite[Proposition~3.1,
p.~82]{Hu} imply that $v(\cdot)$ is upper semi-continuous on $\X$.
Thus, the value function $v(\cdot)$ is
continuous on $\X$. %Following Berge \cite[p.~116]{Ber},
Since $u(\cdot,\cdot)$ is a continuous function on ${\rm
Gr}_{\X}(\Phi)$ and $v(\cdot)$ is a continuous function on $\X$,
the set $ {\rm Gr}_{\X}(\Phi^*)=\{(x,y)\in{\rm Gr}_{\X}(\Phi)\,:\,
u(x,y)-v(x)\le 0\} $ is closed.  According to Theorem~\ref{MT1a},
$\Phi^*(x)\in \K(\Y)$ for all $x\in \X.$

Now we additionally assume that $\Phi$ is upper semi-continuous.
Since $\Phi^*(x)=\Phi^*(x)\cap\Phi(x)$, $x\in\X$, from Berge
\cite[Theorem~7, p.~112]{Ber}, $\Phi^*$ is upper semi-continuous.
\end{proof}

%\section{Examples}

Theorem~\ref{MT3} states that upper semi-continuity of the
set-valued mapping $\Phi^*$ is a necessary condition for upper
semi-continuity of $v$. According to Luque-V\'{a}sques and
Hern\'{a}ndez-Lerma \cite[Theorem~2]{LVHL} (see also
Hern\'{a}ndez-Lerma and Runggaldier \cite[Lemma 3.2(f)]{4}), for
  metric spaces $\X$ and $\Y$, the function $v(\cdot)$ is
lower semi-continuous, if the set-valued mapping
$\Phi^*:\X\to\S(\Y)$ is lower semi-continuous, $u(\cdot,\cdot)$ is
inf-compact in variable $y$ and lower semi-continuous. The
following examples show that lower semi-continuity of the mapping
$\Phi^*$ is not necessary for lower semi-continuity of $v(\cdot)$.

\begin{example}\label{exa2}
{\rm The function $v(\cdot)$ is continuous; the real function
$u(\cdot,\cdot)$ is $\K$-inf-compact on ${\rm Gr}_{\X}(\Phi)$ and
 continuous on
$\X\times\Y$, but it is not inf-compact on ${\rm Gr}_{\X}(\Phi)$;
the set-valued mapping $\Phi:\X\to \S(\Y)$ is continuous; the
set-valued mapping $\Phi^*:\X\to \K(\Y)$ is not lower
semi-continuous. Let $\X=[0,+\infty)$, $\Y=\R$,
${\Phi}(x)=(-\infty,x]$, $u(x,y)=|\min\{x,y+1\}|$, $x\in\X$,
$y\in\Y$. Then
\[{\Phi}^*(x)=\left\{
\begin{array}{ll}
[-1,0],&x=0,\\
\{-1\},&x> 0,
\end{array}
\right.\quad\mbox{and}\quad v(x)\equiv 0.
\]
Therefore, ${\Phi}^*:\X\to \K(\X)$ is not lower semi-continuous.
The function $u(\cdot,\cdot)$ is not inf-compact on ${\rm
Gr}_\X(\Phi)$, since $(x,-1)\in\mathcal{D}_{u(\cdot,\cdot)}(0;{\rm
Gr}_\X(\Phi))$ for each $x\ge 0$.}
\end{example}

The following example is similar to Example~\ref{exa2}, but the
function $u$ is inf-compact on ${\rm Gr}_{\X}(\Phi)$.

\begin{example}\label{exa1}
{\rm The function $v(\cdot)$ is continuous and inf-compact on
$\X$; the real function $u(\cdot,\cdot)$ is inf-compact on ${\rm
Gr}_{\X}(\Phi)$ and continuous on $\X\times\Y$; the set-valued
mapping $\Phi:\X\to \S(\Y)$ is continuous; the set-valued mapping
$\Phi^*:\X\to \K(\Y)$ is not lower semi-continuous. Let
$\X=[0,+\infty)$, $\Y=\R$, ${\Phi}(x)=(-\infty,x]$,
$u(x,y)=|\min\{x,y+1\}|+x$, $x\in\X$, $y\in\Y$. Then
\[
{\Phi}^*(x)=\left\{
\begin{array}{ll}
[-1,0],&x=0,\\
\{-1\},&x> 0,
\end{array}
\right. \quad\mbox{and}\quad v(x)=x.
\]
Therefore, ${\Phi}^*:\X\to \K(\X)$ is not lower semi-continuous.}
\end{example}

The following example shows that continuity properties of $\Phi^*$
may not hold either under the assumptions of Theorem~\ref{MT1} or
under the stronger assumptions of Berge's theorem.

\begin{example}\label{exa3}
{\rm The function $v(\cdot)$ is inf-compact on $\X$; the real
function $u(\cdot,\cdot)$ is inf-compact on ${\rm Gr}_{\X}(\Phi)$,
but it is not upper semi-continuous on ${\rm Gr}_{\X}(\Phi)$; the
set-valued mapping $\Phi:\X\to \K(\Y)$ is continuous; the
set-valued mapping $\Phi^*:\X\to \K(\Y)$ is neither lower
semi-continuous nor upper semi-continuous, and ${\rm
Gr}_\X(\Phi^*)$ is not closed. Let $\X=[0,1]$, $\Y=[-1,1]$,
${\Phi}(x)=\Y$,
\[
u(x,y)=
\left\{
\begin{array}{llll}
0,&x=0&\mbox{and}&y\in[-1,0],\\
y,&x=0&\mbox{and}&y\in(0,1],\\
1-y,&x\in(0,1]&\mbox{and}&y\in[-1,0],\\
1,&x\in(0,1]&\mbox{and}&y\in(0,1].\\
\end{array}\right.
\]
%$x\in\X$ and $y\in\Y$.
Then
\[
{\Phi}^*(x)=
\left\{
 \begin{array}{ll}
   [-1,0], & x=0, \\
   \ [0,1], & x\in(0,1],
    \end{array}
     \right.
     \quad\mbox{and}\quad
     v(x)=\left\{
           \begin{array}{ll}
             0, & x=0,\\
             1, & x\in(0,1].
              \end{array}
               \right.
\]
Therefore, ${\Phi}^*:\X\to \K(\X)$ is neither lower
semi-continuous nor upper semi-continuous, and ${\rm
Gr}_\X(\Phi^*)$ is not closed.}
\end{example}

\vspace{.3cm}
 {\bf Acknowledgements.}
 Research of the first
author was partially supported by NSF grants  CMMI-0900206 and
CMMI-0928490. The authors thank Professor Michael~Zgurovsky for
initiating their research cooperation.

\end{document}